\documentclass[hidelinks,12pt]{amsart}
\usepackage{mathrsfs}
\usepackage{amsfonts, amssymb, amsmath, amscd}
\usepackage{enumerate}
\usepackage{graphicx}
\usepackage{tikz}
\usepackage{pgfplots}
\pgfplotsset{compat=1.14}
\usetikzlibrary{shapes}
\usetikzlibrary{plotmarks}
\usepackage{stmaryrd}
\usepackage{booktabs}
\usepackage{mathtools}
\usepackage{young}
\usepackage{color,soul}
\usepackage{MnSymbol,wasysym}
\usepackage{calligra,mathrsfs}
\usetikzlibrary{matrix,arrows}
\usepackage[balance,small,nohug]{diagrams}

\usepackage[left=.9truein,right=.9truein,bottom=.9truein,top=.9truein]{geometry}

\makeatletter
\def\@tocline#1#2#3#4#5#6#7{\relax
  \ifnum #1>\c@tocdepth 
  \else
    \par \addpenalty\@secpenalty\addvspace{#2}%
    \begingroup \hyphenpenalty\@M
    \@ifempty{#4}{%
      \@tempdima\csname r@tocindent\number#1\endcsname\relax
    }{%
      \@tempdima#4\relax
    }%
    \parindent\z@ \leftskip#3\relax \advance\leftskip\@tempdima\relax
    \rightskip\@pnumwidth plus1em \parfillskip-\@pnumwidth
    #5\leavevmode\hskip-\@tempdima #6\relax
    \dotfill\hbox to\@pnumwidth{\@tocpagenum{#7}}\par
    \nobreak
    \endgroup
  \fi}
\makeatother

\let\oldtocsection=\tocsection
\let\oldtocsubsection=\tocsubsection

\renewcommand{\tocsection}[2]{\hspace{0em}\oldtocsection{#1}{#2}}
\renewcommand{\tocsubsection}[2]{\hspace{1.25em}\oldtocsubsection{#1}{#2}}

\usepackage{hyperref}
\hypersetup{pdfborder = {0 0 0}}

\newcommand{\sheafhom}{\mathscr{H}\text{\kern -3pt {\calligra\large om}}\,}
\newcommand{\sheafext}{\mathscr{E}\text{\kern -3pt {\calligra\large xt}}\,}

\newtheorem{theo}{Theorem}[section]

\newtheorem*{thm*}{Theorem}
\newtheorem{proposition}[theo]{Proposition}
\newtheorem*{proposition*}{Proposition}

\newtheorem*{prop*}{Proposition}

\newtheorem*{remark*}{Remark}
\newtheorem{lemma}[theo]{Lemma}
\newtheorem*{lemma*}{Lemma}

\newtheorem*{cor*}{Corollary}

\newtheorem*{claim*}{Claim}

\newtheorem*{details*}{Details}

\newtheorem*{recall*}{Recall}

\newtheorem*{ass*}{Assumption}

\newtheorem*{conj*}{Conjecture}

\newtheorem*{intprob*}{The Interpolation Problem}

\theoremstyle{definition}
\newtheorem{definition}[theo]{Definition}
\newtheorem*{definition*}{Definition}

\newtheorem*{deff*}{Definition}

\newtheorem*{problem*}{Problem}

\newtheorem*{prob*}{Problem}

\newtheorem*{ques*}{Question}
\newtheorem{note}[theo]{Note}
\newtheorem*{note*}{Note}

\makeatletter
\newtheorem*{rep@theorem}{\rep@title}
\newcommand{\newreptheorem}[2]{%
\newenvironment{rep#1}[1]{%
 \def\rep@title{#2 \ref{##1}}%
 \begin{rep@theorem}}%
 {\end{rep@theorem}}}
\makeatother
\newtheorem{theorem}[theo]{Theorem}
\newtheorem*{theorem*}{Theorem}
\newreptheorem{theorem}{Theorem}
\newcommand{\pr}{\operatorname{pr}}

\begin{document}
\title{The geometry of Hilbert schemes of two points on projective space}
\author{Tim Ryan}
\address{Department of Mathematics, University of Michigan}
\email{rtimothy@umich.edu}
\begin{abstract} In this paper, we give three bases for the cohomology groups of the Hilbert scheme of two points on projective space. Then, we use these bases to compute all effective and nef cones of higher codimensional cycles on the Hilbert scheme. Next, we compute the class in one of these bases of the Chern classes of tautological bundles coming from line bundles. Finally, we provide an application of these results to the degrees of secant varieties of complete intersections.
\end{abstract}
\maketitle
\setcounter{tocdepth}{2}
\vspace{-.5in}
\section{Introduction}
\medskip\noindent
In this paper, we will generalize several results known on Hilbert schemes of points on the complex projective plane to Hilbert schemes of two points on any complex projective space.
In particular, we will gives geometric bases for the Chow groups, compute the effective and nef cones of cycles, give the classes of Chern classes of tautological line bundles, and provide an application of these results by computing the degree of the secant variety to any  complete intersection subvariety of projective space that is not $1$-defective.

\medskip\noindent
Hilbert schemes of points on the projective plane have been extensively studied, see for example the books \cite{N,N2,G}. 
They have proved interesting in their own right and also as a base case or testing ground for more general results.
One limit on extending this usefulness to other Hilbert schemes has been the wild misbehavior of Hilbert schemes (of points) on higher dimensional projective spaces \cite{V}.
The new results of the paper by Skjelnes and Smith \cite{SS} narrows this obstruction by characterising which Hilbert schemes on projective space are smooth.
Their result gives the six infinite families of smooth Hilbert schemes on (each) projective space (and a seventh trivial family).
On $\mathbb{P}^n$, one of these families are the Hilbert schemes with Hilbert polynomials\[p_{r,\lambda_{r-1}}(t) = \left(\sum_{i=1}^{r-2} \binom{t+n-i}{n-1}\right) +\binom{t+\lambda_{r-1}-r-1}{\lambda_{r-1}-1}+1\] 
where $r\geq 2$ and $\lambda_{r-1} \geq 1$ are integers.
Setting $r = 2$ and $\lambda_{r-1}=1$ gives the Hilbert scheme of two points on projective space, denoted $\mathbb{P}^{n[2]}$.
In this sense, the Hilbert scheme of two points is the base case for either of the two sets of base cases for this family.
This paper aims to completely study these as a basis for future study of this entire family.

\medskip\noindent
We start by considering the Chow ring of $\mathbb{P}^{n[2]}$.
The dimension of these rings as complex vector spaces can be computed to be $3\binom{n+1}{2}$  \cite{C,C2}.  
Our first result gives three bases for these vector spaces with respect to two fixed flags, $\mathbb{P}^0 = F_0 \subset \cdots\subset F_n = \mathbb{P}^n$ and $\mathbb{P}^0 = F_0' \subset \cdots\subset F_n' = \mathbb{P}^n$ such that $\mathrm{dim}(F_i\cap F_j') = i-1$ if $i\leq j$. 
To state that result, we first must define the following classes, which we define set-theoretically and take the reduced induced scheme structure.
\begin{align*}
    A_{i,j} =& \{Z \in \mathbb{P}^{n[2]} : \mathrm{supp}(Z) \cap F_i \neq \varnothing \text{ and } Z \subset F_j\}\\    
    A_{i,j}' =& \{Z \in \mathbb{P}^{n[2]} :\mathrm{dim}\left(\mathrm{Span}(Z,F_{i})\right) \leq i+1, Z \subset F_j, \text{ and }Z \text{ is nonreduced}\}
\end{align*}
\begin{align*}
    B_{i,j} =& \{Z \in \mathbb{P}^{n[2]} : \mathrm{supp}(Z) \subset F_i, Z \subset F_{j+1}, \text{ and } Z \text{ is nonreduced}\}\\
    B_{i,j}' =& \{Z \in \mathbb{P}^{n[2]} : Z\cap F_i \neq \varnothing,  Z \cap F_j' \neq \varnothing, \text{ and } Z \subset F_i\cup F_j'\}\\
    C_{i,j} =& \{Z \in \mathbb{P}^{n[2]} : \mathrm{dim}\left(\mathrm{Span}(Z,F_{i-1})\right) \leq i \text{ and } Z \subset F_j\}\\
\end{align*}
Note, by the \textit{span} of a set of subschemes we mean the smallest dimensional linear subspace containing all of them.

\begin{theorem}
Define the sets $A =\{A_{i,j}:0\leq i <j \leq n\}$, $A' =\{A_{i,j}':0\leq i < j \leq n\}$, $B = \{B_{i,j}:0\leq i \leq j <n\}$, $B' = \{B_{i,j}':0\leq i \leq j <n\}$, and $C= \{C_{i,j}:0< i \leq j \leq n\}$.
Similarly, define the unions of these sets $\mathcal{B}_{\text{BB}} = A\cup B\cup C$, $\mathcal{B}_{\text{ES}} = A' \cup B \cup C$, and $\mathcal{B}_{\text{MS}} = A \cup B' \cup C$.
Then $\mathcal{B}_{\text{BB}}$, $\mathcal{B}_{\text{ES}}$, and $\mathcal{B}_{\text{MS}}$ are each a basis for the Chow ring $A\left(\mathbb{P}^{n[2]}\right)$.
\end{theorem}

\medskip\noindent
We will refer to $\mathcal{B}_{\text{BB}}$, $\mathcal{B}_{\text{ES}}$, and $\mathcal{B}_{\text{MS}}$ as the BB basis (for Bia\l ynicki-Birula), the ES basis (for Ellingsrud and Str\o mme), and the MS basis (for Mallavibarrena and Sols), respectively, and we will refer to the subset of a basis $\mathcal{B}_*$ which is the codimension $k$-cycles by $\mathcal{B}_*^k$ and which is the dimension $k$-cycles by $\mathcal{B}_{*,k}$.
Three bases may seem excessive, but they are each natural for different reasons. 
The BB basis is natural from the perspective of a $\mathbb{C}^{*}$-action as its elements are the closures of the loci which limit to each fixed point.
The perspective from which the ES and MS bases are natural is made clear by the following theorem which computes the effective and nef cones of cycles on $\mathbb{P}^{n[2]}$.
In order to state the theorem, we note that \[\mathrm{dim}\left(A_{i,j}\right) =\mathrm{dim}\left(A_{i,j}'\right) =\mathrm{dim}\left(B_{i,j}\right) = \mathrm{dim}\left(B_{i,j}'\right) =\mathrm{dim}\left(C_{i,j}\right) = i+j,\]

\begin{theorem}
\label{thm: cones}
The effective cone of dimension $k$ cycles, denoted $\mathrm{Eff}_k\left(\mathbb{P}^{n[2]}\right) \subset N_{k}(\mathbb{P}^{n[2]})$, is spanned by  $\mathcal{B}_{\mathrm{ES},k}$, and the nef cone of codimension $k$ cycles, denoted $\mathrm{Nef}^k\left(\mathbb{P}^{n[2]}\right)  \subset N^{k}(\mathbb{P}^{n[2]})$, is spanned by $\mathcal{B}_{\mathrm{MS}}^k$.
\end{theorem}

\medskip\noindent
These bases are also useful in computing the Chern classes of tautological bundles, where tautological bundles are the push-forward to the Hilbert scheme of bundles on the universal family pulled back from $\mathbb{P}^n$.
Formally, the universal family, which can be defined set theoretically as \[\mathbb{P}^{n[1,2]} =\{(p,Z) \subset \mathbb{P}^n \times \mathbb{P}^{n[2]}: p\in \mathrm{supp}(Z)\},\]
comes equipped with projections to the Hilbert scheme and to projective space.
\[\begin{diagram}
&& \mathbb{P}^{n[2,1]} && \\
& \ldTo^{\pr_2} & & \rdTo^{\pr_1} &\\
\mathbb{P}^{n[2]} &&&& \mathbb{P}^{n}
\end{diagram}\]
Given a coherent sheaf $\mathcal{F}$ on projective space, the \textit{tautological sheaf} is the bundle \[V^{[2]} = (\pr_{2})_*(\pr_1^*(\mathcal{F}))\] that is gotten by pulling the sheaf back to the universal family and pushing it forward to the Hilbert scheme.
When $\mathcal{F}$ is a line bundle, these have been of extensive interest, e.g. \cite{L,MOP,Vo,MOP2,J}.
As $\pr_2$ is a finite flat morphism of degree 2, $\mathcal{F}^{[2]}$ is a rank 2 vector bundle in that case.
Thus, it only has two non-zero Chern classes, which we compute in the MS basis.
\begin{theorem}\label{thm: chern}
In terms of the MS basis, we have 
\[c_1\left(\mathcal{O}_{\mathbb{P}^n}(d)^{[2]}\right) = (d-1)A_{2n-1,2n}+C_{2n-1,2n} \text{ and}\] 
\[c_2\left(\mathcal{O}_{\mathbb{P}^n}(d)^{[2]}\right) =\binom{d}{2}B_{2n-1,2n-1}'+dC_{2n-1,2n-1}\]
for $n>1$.
If $n=1$, the classes are the same without the second terms, as those classes no longer exist.
\end{theorem}

\medskip\noindent
We can apply these results to the study of secant varieties.
Recall that the (1st)-secant variety of a subvariety $X$ of projective space is the closure of the union of the secant lines to that variety, denoted  by $\mathrm{Sec}(X)$.
When $\dim(X)=m$ and $X \subset \mathbb{P}^n$ with $2m+1<n$, the expected dimension of $\mathrm{Sec}(X)$ is $2m+1$.
When the secant variety has the expected dimension, it is natural to ask what is its degree.
Indeed, studying the degree of the secant variety has a long and extensive history dating back to at least Giambelli and continuing to the present, see for example \cite{Gi,B,T,D,La,C-J,Z,CR}.
In particular, Theorem 19 of \cite{La} computes $\mu_1(X) \dot \deg(\mathrm{Sec}(X))$ when $X$ is smooth and connected, where $\mu_1(X)$ is the 1-st secant order of $X$, i.e. the number of secant lines to $X$ passing through a general point of $\mathrm{Sec}(X)$.
Results for singular curves and surfaces are also known \cite{AMS,AMS2,Ca}
Using our results on the Hilbert scheme, we can remove the hypotheses of smooth and connected from that theorem, but at the expense of only considering complete intersections.

\begin{theorem}\label{thm: sec}
Let $X$ be a subvariety of $\mathbb{P}^n$ of dimension $m$ with $2m+1<n$ which is not $1$-defective and which is the complete intersection of hypersurfaces of degree $d_1$, $\dots$, and $d_{n-m}$.
The degree of the secant variety, $\mathrm{Sec}(X)$, is \[\deg\left(\mathrm{Sec}(X)\right) = \frac{1}{\mu_1(X)} \sum_{k=m+1}^{n-m} \left(\prod_{j \in \{i_1,\dots,i_k\}} \left(\prod_{l\not \in \{i_1,\dots,i_k\}} \binom{d_{j}}{2} d_l 2^{k-1-m}\right)\right)\] 
where the products are taken over all choices of $k$ integers from the set $\{1,\dots,n-m\}$.
\end{theorem}

\medskip\noindent
The presence of $\mu_1(X)$ is not a large obstruction as it is expected to be 1 for most varieties.

\medskip\noindent
A similarly argument can be run to remove the smoothness hypothesis of the results concerning the number of tangent lines through a general point of projective space when $2m+1=n$, see \cite{PS,HR}.

\medskip\noindent
The paper is organized as follows. 
Section 2 defines the sets which will form the bases and shows that the set $\mathcal{B}_{\text{BB}}$ is the basis given by the $\mathbb{C}^*$-action.
Section 3 computes the complementary codimension intersection matrix of the sets $\mathcal{B}_{\text{ES}}$ and $\mathcal{B}_{\text{MS}}$ to compute the nef and effective cones of cycles on $\mathbb{P}^{n[2]}$ and show that they are bases.
Section 4 computes the Chern classes of tautological bundles on that Hilbert scheme coming from line bundles on projective space.
Finally, Section 5 discusses an application of these results to secant varieties.

\medskip\noindent
The author would like to thank Gregory Taylor, Alexander Stathis, and C\'esar Lozano Huerta for helpful discussions.

\section{Bases for the Chow groups}
\medskip\noindent
In this section, we will describe the three sets of geometric loci, $\mathcal{B}_{\text{BB}}$, $\mathcal{B}_{\text{ES}}$, and $\mathcal{B}_{\text{MS}}$, and show that $\mathcal{B}_{\text{BB}}$ is the basis for the Chow groups of $\mathbb{P}^{n[2]}$ made by the BB-cells.
Note, the Chow groups are also the cohomology groups as the Hilbert scheme is a smooth rational variety.

\subsection{The BB basis}
\medskip\noindent
In order to compute those bases, we first compute the dimension of the Chow groups as complex vector spaces.
By Bia{\l}ynicki-Birula \cite{BB,BB2}, this dimension is the number of fixed points of the inherited $\mathbb{C}^*$-action; note, there is obviously more than one such action, but any action $(x_0,\dots,x_n) \mapsto (\lambda^{a_0} x_0,\dots, \lambda^{a_n} x_n)$ such that $a_0<<\cdots <<a_n$ will work for our purposes.
The fixed points of this action are precisely the monomial ideals with constant Hilbert polynomial equal to two.
In order to count the monomial ideals, we first consider the general form of any ideal of two points.
As any two points are contained in a line, but not in a single point, we can see that the ideal must contain $n-1$ linear generators and a single quadratic generator.
As the only linear monomials are the variables and the only quadratic monomials are of the form $x_i^2$ or $x_ix_j$, the monomial ideals are of one the following three forms where $i<j$:
\[I_{i,j}=(x_ix_j,x_0,\cdots,x_{i-1},\hat{x_i},x_{i+1},\cdots,x_{j-1},\hat{x_j},x_{j+1},\cdots, x_n),\]
\[J_{i,j} = (x_0,\cdots,x_{i-1},x_{i+1},\cdots,x_{j-1},x_j^2,x_{j+1},\cdots, x_n)\text{, or}\]
\[K_{i,j} = (x_0,\cdots,x_{i-1},x_i^2,x_{i+1},\cdots,x_{j-1},\hat{x_j},x_{j+1},\cdots, x_n).\]

%
\medskip\noindent
It is straightforward to count that there are $\binom{n+1}{2}$ monomial ideals of each type.
Thus, we have shown the following lemma.
\begin{lemma} $\mathrm{dim}_{\mathbb{C}}\left( A\left(\mathbb{P}^{n[2]}\right)\right) = 3\binom{n+1}{2}$.
\end{lemma}
\noindent This can also be derived from the generating functions given in \cite{C2}.

\medskip\noindent
Furthermore, Bia{\l}ynicki-Birula \cite{BB2} shows that the cycles giving a basis for the Chow ring are the classes of (the closure of the) cells of the cellular decomposition of the Hilbert scheme induced by the action.
The cells are loci which limit to each fixed point as the parameter $\lambda$ goes to zero.
These are sometimes known as \textit{Bia{\l}ynicki-Birula cells}.
These are related to the \textit{Gr\"obner cells} of Conca and Valla \cite{CV}.
We want to describe each cell geometrically.
In order to do that, let us consider a general ideal $I$ of two points which in the cell corresponding to $I_{i,j}$, $J_{i,j}$, or $K_{i,j}$.
Then we may simplify the generators of the ideal into the following form
\[I=(ax_i^2+bx_ix_j+cx_j^2,x_0-a_0x_i-b_0x_j,\cdots,x_{i-1}-a_{i-1}x_i-b_{i-1}x_j,
x_{i+1}-b_{i+1}x_j,\cdots,x_{j-1}-b_{i-1}x_j,
x_{j+1},\cdots, x_n).\]
If $a \neq 0$, then this ideal limits to $K_{i,j}$.
There are $2+2i+(j-1-i) = i+j+1$ dimensions of such ideals, so the cell containing $K_{i,j}$ is dimension $i+j+1$. 
Denote its class by $C_{i+1,j}$.
If $a =0$ and $b\neq 0$, then this ideal limits to $I_{i,j}$.
There are $1+2i+(j-1-i) = i+j$ dimensions of such ideals so the cell containing containing $I_{i,j}$ is dimension $i+j$. 
Denote its class by $A_{i,j}$.
Finally, if $a=b=0$, then this ideal limits to $J_{i,j}$.
There are $2i+(j-1-i) = i+j-1$ dimensions of such ideals so the cell containing containing $J_{i,j}$ is dimension $i+j-1$.
Denote its class by $B_{i,j-1}$.

\medskip\noindent
We want to describe this cells more geometrically.
To formally do this, fix two flags $\mathbb{P}^0 = F_0 \subset \cdots\subset F_n = \mathbb{P}^n$ 
where the $F_i$ is spanned by the points $e_0,\cdots,e_i$ and $\mathbb{P}^0 = F_0' \subset \cdots\subset F_n' = \mathbb{P}^n$ 
where the $F_i'$ contains is defined by the points $e_n,e_0,\cdots,e_{i-1}$.

\medskip\noindent
Geometrically, the cell $A_{i,j}$ is the locus where the pair of points intersects an $i$-plane and lies inside a $j$-plane which contains that $i$-plane.
Set-theoretically, we write  
\[A_{i,j} = \{Z \in \mathbb{P}^{n[2]} : \mathrm{supp}(Z) \cap F_i \neq \varnothing \text{ and } Z \subset F_j\}\]
for $0\leq i< j \leq n$.
We can see this from the form of the ideal $I$ that limits to $I_{i,j}$.
It is clear that the pair of points is contained in the $j$-plane defined by $x_{j+1} = \dots x_n =0$ since those terms are in the ideal.
Next, it is also clear that the pair of points intersect the $i$-plane defined by $x_{i+1} = \dots x_n =0$ since each of those variables is a zero-divisor in the polynomial ring modded out by the ideal, which means that they vanish at an associated point of the ideal, i.e. at one of the points of the support.
As mentioned above, $A_{i,j}$ is dimension $i+j$ which can be seen geometrically as you can pick any one point of $F_i$ and any one point of $F_j$.
Putting this last two facts together, we see that there are $\mathrm{min}(\lceil \frac{k}{2}\rceil,\lceil n-\frac{k}{2}\rceil)$ cycles $A_{i,j}$ of dimension $k$ for $0\leq k \leq 2n$.
Note, summing over all dimensions, we see that there are $\binom{n+1}{2}$ cycles $A_{i,j}$.

\medskip\noindent
Geometrically, $B_{i,j}$ is the locus where the pair of points are a double point whose support is on an $i$-plane and which is contained in a $(j+1)$-plane which contains that $i$-plane.
Set theoretically, we write 
\[B_{i,j} = \{Z \in \mathbb{P}^{n[2]} : \mathrm{supp}(Z) \subset F_i, Z \subset F_{j+1}, \text{ and } Z \text{ is nonreduced}\}.\]
for $0\leq i \leq j < n$.
We can see this from the form of the ideal $I$ that limits to $J_{i,j+1}$.
The presence of $x_{j+1}^2$ as a generator forces the pair to be non-reduced.
Next, it is clear that the pair of points is supported on the $i$-plane defined by $x_{i+1} = \dots x_n =0$ since those terms are in the radical of the ideal.
Finally, it is also clear that the pair of points is contained in the $(j+1)$-plane defined by $x_{j+2} = \dots x_n =0$ since those terms are in the ideal.
$B_{i,j}$ is dimension $i+j$ as you can pick the support anywhere in $F_i$ and then pick any tangent direction in  $F_{j+1}$. 
Putting this last two facts together, we see that there are $\mathrm{min}(\lceil\frac{k+1}{2}\rceil,\lceil n-\frac{k+1}{2}\rceil)$ cycles $B_{i,j}$ of dimension $k$ for $0\leq k \leq 2n$.
Note, summing over all dimensions, we see that there are $\binom{n+1}{2}$ cycles $B_{i,j}$.

\medskip\noindent
Geometrically, $C_{i,j}$ is the cycle where the pair of points are linearly dependent with $i$ points in general position in a $j$-plane and are contained in that $j$ plane. 
Set theoretically, we write 
\[C_{i,j} = \{Z \in \mathbb{P}^{n[2]} : \mathrm{dim}\left(\mathrm{Span}(Z,F_{i-1})\right) \leq i \text{ and } Z \subset F_j\}\]
for $0< i\leq  j \leq n$ .
We can see this from the form of the ideal $I$ that limits to $K_{i-1,j}$.
It is clear the pair of points is contained in the $j$-plane defined by $x_{j+1} = \dots x_n =0$.
Secondly, knowing one of the pair of the points (or the support if it is non-reduced), determines the coefficients on every linear polynomial in the ideal containing $x_k$ with $k>i$ so the points are linearly dependent with the coordinate points $e_k$, for $k<i$ as desired.
$C_{i,j}$ is dimension $i+j$ as you can pick the first point anywhere in $F_j$ and then pick the second point in the $i$-plane determined by the support and the $i$ general points. 
Putting this last two facts together, we see that there are $\mathrm{min}(\lceil \frac{k-1}{2}\rceil,\lceil n- \frac{k-1}{2}\rceil)$ cycles $C_{i,j}$ of dimension $k$ for $0\leq k \leq 2n$.
Note, summing over all dimensions, we see that there are $\binom{n+1}{2}$ cycles $C_{i,j}$.

\medskip\noindent
Since this cycles are the cells of the Bia{\l}ynicki-Birula cellular decomposition, they form a basis for the Chow ring.
In other words, if we define the sets $A =\{A_{i,j}:0\leq i <j \leq n,\}$,, $B = \{B_{i,j}:0\leq i \leq j <n\}$, and $C= \{C_{i,j}:0< i \leq j \leq n\}$ and \[\mathcal{B}_{\text{BB}} = A \cup B \cup C,\]
then the following proposition is immediate.

\begin{proposition}
$\mathcal{B}_{\text{BB}}$ is a basis for $A\left(\mathbb{P}^{n[2]}\right)$, $\mathrm{dim}\left(A^{2n}\left(\mathbb{P}^{n[2]}\right)\right)=\mathrm{dim}\left(A^0\left(\mathbb{P}^{n[2]}\right)\right) = 1$, and \[\mathrm{dim}\left(A^k\left(\mathbb{P}^{n[2]}\right)\right) = \mathrm{min}\left(\lceil\frac{k}{2}\rceil,\lceil n-\frac{k}{2}\rceil\right)+\mathrm{min}\left(\lceil\frac{k+1}{2}\rceil,\lceil n-\frac{k+1}{2}\rceil\right)+\mathrm{min}\left(\lceil\frac{k-1}{2}\rceil,\lceil n-\frac{k-1}{2}\rceil\right)\] for $1\leq k\leq 2n-1$.
\end{proposition}
\medskip\noindent
We will call $\mathcal{B}_{\text{BB}}$ the BB-basis.


\subsection{Other Families}
\medskip\noindent
We will also need two other families of cycles.
The first family is where the pair of points are a double point which is linearly dependent (as a double point) with $i+1$ points in general position in a $j$-plane and are contained in that $j$ plane. 
Set theoretically, we write
\[A_{i,j}' = \{Z \in \mathbb{P}^{n[2]} :\mathrm{dim}\left(\mathrm{Span}(Z,F_{i})\right) \leq i+1, Z \subset F_j, \text{ and }Z \text{ is nonreduced}\}.\]
for $0\leq i< j \leq n$. 
$A_{i,j}'$ is dimension $i+j$ as you can pick the support anywhere in $F_j$ and then pick any tangent direction in the $i+1$-plane determined by the support and the $i+1$ general points. 
The same counts as for $A_{i,j}$ give the number of these cycles in each dimension.

\medskip\noindent
The second family is where the pair of points intersect an $i$-plane and a $j$-plane, which intersect in an $(i-1)$-plane with have $i\leq j$ and are contained in the union of the two.
Set theoretically, we write
\[B_{i,j}' = \{Z \in \mathbb{P}^{n[2]} : Z\cap F_i \neq \varnothing,  Z \cap F_j' \neq \varnothing, \text{ and } Z \subset F_i\cup F_j'\}.\]
for $0\leq i\leq j < n$. 
$B_{i,j}'$ is dimension $i+j$ as you can pick any one point of $F_i$ and any one point of $F_j'$. 
The same counts as for $B_{i,j}$ give the number of these cycles in each dimension.

\medskip\noindent
Similarly to the three families of BB cells, we define the sets $A' =\{A_{i,j}':0\leq i < j \leq n\}$ and $B' = \{B_{i,j}':0\leq i \leq j <n\}$.
We can then define the two sets 
\begin{align}
\label{MS} \mathcal{B}_{MS} &= A \cup B' \cup C\text{ and}\\
\label{ES} \mathcal{B}_{ES} &= A' \cup B \cup C.
\end{align}
In the following section, we will show these are bases of the Chow ring and as such we will refer to them as the MS basis and ES basis from now on.

\begin{note}
By convention, we list cells in dimension $k$ by increasing order on the first index, but we list cells in codimension $k$ by decreasing order on the first index.
\end{note}

\section{Effective and Nef Cones of Cycles}
\medskip\noindent
In this section, we compute the effective and nef cones of cycles on $\mathbb{P}^{n[2]}$.
To do that, we want to show that $\mathcal{B}_{\text{MS}}$ and $\mathcal{B}_{\text{ES}}$ are dual with respect to the intersection product (by which we mean their intersection matrix in each dimension is diagonal ).
As a consequence, we will see that they are bases and that they span the extremal rays of the nef and effective cones of cycles, respectively.

\subsection{Intersections}
\medskip\noindent
Let's first show which of the relevant intersections are zero.
This will depend on the indices of the cycles.
\begin{definition}
We say that two cycles $A_{i,j}/ B_{i,j}/C_{i,j}$ and $A_{k,l}/ B_{k,l}/C_{k,l}$ have \textit{complementary indices} if $k = n-i$ and $l = n-j$.
\end{definition}

\medskip\noindent
We can now state the proposition.

\begin{proposition}\label{prop: zeros}
Let $0\leq m \leq n$.
Then we have the intersections 
\begin{align*}
    &A_{i,j}\cdot B_{k,l} = 0,
    &\hspace{.7in}A_{i,j}\cdot C_{k,l} = 0,
    &\hspace{.7in}A_{i,j}'\cdot B_{k,l}' = 0,\\ 
    &A_{i,j}'\cdot C_{k,l} = 0,
    &\hspace{.7in}B_{i,j}\cdot B_{k,l}' = 0 ,
    &\hspace{.7in}C_{i,j}\cdot C_{k,l} = 0
\end{align*}
when both classes exist and have complementary codimensions $m$ and $2n-m$, respectively.
Similarly, we have 
\begin{align*}
    &A_{i,j}\cdot A_{k,l}' = 0,
    &\hspace{.7in}B_{i,j}\cdot C_{k,l} = 0,
    &\hspace{.7in}B_{i,j}'\cdot C_{k,l} = 0
\end{align*} 
unless they have complementary indices.
\end{proposition}

\begin{proof}
\medskip\noindent
Let us first consider the intersections which are always zero.
Recall that $A_{i,j}$ is the locus of pairs of points which intersect a fixed general $i$-plane and are contained in a fixed general $j$-plane containing the $i$-plane and that $B_{k,l}$ is the locus of nonreduced pairs of points supported on a fixed general $k$-plane and contained in an $(l+1)$-plane.
In order for a pair of points to lie in both of these loci, we must have that $i+k\geq n$ and $j+l+1\geq n+1$.
Since $i< j$, $k\leq l$, and $i+j+k+l = 2n$ (as they are complementary codimension), we see that $i+k=n$ and $j+l=n$.
Then $n-i = k \leq l = n-j$, so $j \leq i$. 
This is not possible as $i< j$, so $A_{i,j}$ and $B_{k,l}$ are disjoint, so \[A_{i,j} \cdot B_{k,l} = 0.\]

\medskip\noindent
Similarly, recall that $C_{k,l}$ is the locus of pairs of points which are linearly dependent with $k$ points and are contained in an $l$-plane which contains those points.
In order for a pair of points to lie in $A_{i,j}$ and $C_{k,l}$, we see that $i+l\geq n$ and $j+k \geq n+1$ since the $k$-plane must intersect the $j$-plane in at least 2 points.
Then $2n \geq i+j+k+l \geq 2n+1$ which is impossible so we get  \[A_{i,j}\cdot C_{k,l} = 0.\]

\medskip\noindent
Further, recall that $A_{i,j}'$ is the locus of nonreduced pairs of points which are linearly dependent with $(i+1)$ points and are contained in an $j$-plane which contains those points and that $B_{k,l}'$ is the locus of pairs of points which intersect a fixed general $k$-plane and a fixed general $l$-plane intersecting the $k$-plane in a dimension $(k-1)$-plane and contained in their union.
In order for a pair of points to lie in $A_{i,j}'$ and $B_{k,l}'$, we see that $(i+1)+(k-1)\geq n+1$.
Since $i< j$ and $k\leq l$, we see that $i+j+k+l\geq 2n+2$.
Since $i+j+k+l = 2n$ (as they are complementary codimension), 
\[A_{i,j}'\cdot B_{k,l}' = 0.\]

\medskip\noindent
Next, consider the intersection of $A_{i,j}'$ and $C_{k,l}$.
In order for a pair of points to lie in $A_{i,j}'$ and $C_{k,l}$, we see that $i+1+k\geq n+1$.
Since $i< j$, $k\leq l$, and $i+j+k+l = 2n$ (as they are complementary codimension), we see that $i+k=n$ and $j+l=n$.
Then $n-i = k \leq l = n-j$, so $j \leq i$. 
Since $i<j$,
\[A_{i,j}' \cdot C_{k,l} = 0.\]

\medskip\noindent
Then consider the intersection of $B_{i,j}$ and $B_{k,l}'$.
If a pair of points lies in $B_{i,j}$ and $B_{k,l}'$, then $i+(k-1)\geq n$.
Since $i< j$ and $k\leq l$, we see that $i+j+k+l\geq 2n+2$.
Since $i+j+k+l = 2n$ (as they are complementary codimension), this is impossible so they are disjoint and we get \[B_{i,j}\cdot B_{k,l}' = 0.\]

\medskip\noindent
Now, consider the intersection of $C_{i,j}$ and $C_{k,l}$.
If a pair of points lies in $C_{i,j}$ and $C_{k,l}$, then $i+k\geq n+1$.
Since $i< j$ and $k\leq l$, we see that $i+j+k+l\geq 2n+2$.
Since $i+j+k+l = 2n$ (as they are complementary codimension), this is impossible so they are disjoint and we get \[C_{i,j}\cdot C_{k,l} = 0.\]

\medskip\noindent
Let's now consider the intersections which are not always zero.
In order for a pair of points to lie in $A_{i,j}$ and $A_{k,l}'$, we see that $i+l\geq n$ and $k+1+j\geq n+1$.
Since $i< j$, $k\leq l$, and $i+j+k+l = 2n$ (as they are complementary codimension), we see that $i+l=n$ and $j+k=n$.
Thus, the intersection is disjoint unless $i+l = n$. 
If $i+l=n$, then geometrically we can see a pair in the intersection so the intersection is positive.

\medskip\noindent
Similarly, in order for a pair of points to lie in $B_{i,j}$ and $C_{k,l}$, we see that $i+l\geq n$ and $k+j+1\geq n+1$.
By similar argument to the previous case, the intersection is disjoint unless $i+l = n$. 
If $i+l=n$, then geometrically we can see a pair in the intersection so the intersection is positive.

\medskip\noindent
Finally, in order for a pair of points to lie in $B_{i,j}'$ and $C_{k,l}$, we see that $i+l\geq n$ and $k+j\geq n$.
By similar argument to the previous case, the intersection is disjoint unless $i+l = n$. 
If $i+l=n$, then geometrically we can see a pair in the intersection so the intersection is positive.
\end{proof}

\medskip\noindent
Next, we can use these intersection to compute the nef and effective cones of cycles, i.e. Theorem \ref{thm: cones}.

\begin{reptheorem}{thm: cones}
The extremal rays of $\mathrm{Eff}_k\left(\mathbb{P}^{n[2]}\right)$ are spanned by $\mathcal{B}_{\text{ES},k}$ and  the extremal rays of $\mathrm{Nef}^{k}\left(\mathbb{P}^{n[2]}\right)$ are spanned by $\mathcal{B}_{\text{MS}}^{k}$ for $0\leq k \leq 2n$ where $\mathcal{B}_{\text{MS}}^k$ and  $\mathcal{B}_{\text{ES},k}$ are as in Equations \eqref{MS} and \eqref{ES}, respectively.
\end{reptheorem}

\begin{proof}
\medskip\noindent
We start by noting why the cycles in $\mathcal{B}_{\mathrm{ES},k}$ are effective and why the cycles in $\mathcal{B}_{\mathrm{MS}}^{k}$ are nef.
Since we have given loci giving each cycle in $\mathcal{B}_{\mathrm{ES},k}$, these cycles are effective.
Since the generic point of any cycle in $\mathcal{B}_{\mathrm{MS}}^{k}$ is two reduced points and therefore intersects the two orbits of the inherited $\mathrm{PGL}(n+1)$-action on $\mathbb{P}^{n[2]}$ in the correct dimension, they are all nef \cite{MeS}.

\medskip\noindent
In order to show that these cycles span the respective cones, it only remains to show that these two bases are dual to each other.
Let $A^k$ be the cycles in $A$  of codimension $k$ listed by lexiographic order on their subscript ${i,j}$ and let $A_k$ be the cycles in $A$ of dimension $k$ listed by anti-lexiographic order on their subscript ${i,j}$.
Define this similarly for $B$, $C$, $A'$, and $B'$.
Given this notation, the previous proposition gives the intersection matrix of $\mathcal{B}_{\mathrm{MS}}^k$ and $\mathcal{B}_{\mathrm{ES},k}$, see Table \ref{fig: MS/ES intersections}.
\begin{table}
    \centering
\[\begin{array}{c|c|c|c|}
     & A^k & B^{'k} & C^k\\ \hline
    A'_k & D & 0 & 0 \\ \hline
    C_k & 0 & D' & 0 \\ \hline
    B_k & 0 & 0 & D'' \\ \hline
\end{array}\]
    \caption{The intersection matrix of complementary codimension cycles in the ES and MS bases where $D$, $D'$, and $D''$ are full rank diagonal matrices.}
    \label{fig: MS/ES intersections}
\end{table}
Thus, they are dual and, therefore, span the nef and effectives cones respectively.
\end{proof}

\section{Chern Classes of Tautological Line Bundles}
\medskip\noindent
In this section, we will compute the Chern classes of tautological bundles on $\mathbb{P}^{n[2]}$. 

\medskip\noindent
Let us first recall the construction of tautological vector bundles on the Hilbert scheme.
Recall that the universal family over the Hilbert scheme, denoted $\mathbb{P}^{n[1,2]}$, comes with the projection maps
$\pr_1:\mathbb{P}^{n[1,2]} \to \mathbb{P}^{n}$ and $\pr_2:\mathbb{P}^{n[1,2]} \to \mathbb{P}^{n[2]}$
Given a sheaf $\mathcal{F}$ on $\mathbb{P}^n$, we define the tautological sheaf $\mathcal{F}^{[2]}$ on $\mathbb{P}^{n[2]}$ to be the sheaf $\left(\pr_{2}\right)_* \left(\pr_1\right)^*\left(\mathcal{F}\right)$.
Since $\pr_2$ is a finite flat morphism of degree 2, $\mathcal{O}_{\mathbb{P}^n}(d)^{[2]}$ is a vector bundle of rank 2 on $\mathbb{P}^{n[2]}$ so computing its Chern classes reduces to computing its first and second Chern classes.

\medskip\noindent
In order to compute those Chern classes, we need to intersect them with the MS basis and combine those intersections with the intersection matrix of $\mathcal{B}_{\mathrm{MS},k}$ with $\mathcal{B}_{\mathrm{MS}}^k$.
The intersection matrix of those two sets is the content of the next lemma.

\begin{table}
    \centering
\[\begin{array}{c|c|c|c|}
     & A^k & B^{'k} & C^k\\ \hline
    A_k & I & I & 0 \\ \hline
    B^{'}_k & I & J & I \\ \hline
    C_k & 0 & I & 0 \\ \hline
\end{array}\]
    \caption{The intersection matrix for complementary codimension cycles in the MS basis where each $0$ is a zero matrix, each $I$ is a matrix which is one when the cells have complementary indices and $0$ otherwise, and $J$ is the matrix which is $0$ unless the cycles have complementary indices, in which case it is $2$ for $B'_{i,i}\cdot B'_{n-i,n-i}$ and $1$ otherwise.}
    \label{table: MS intersections}
\end{table}
\begin{lemma}
\label{lem: MS intersections}
The complementary codimension intersections of the MS basis are as in Table \ref{table: MS intersections}.
\end{lemma}

\begin{proof}
The zero matrices follow from Prop. \ref{prop: zeros}. We now want to consider the non-zero block matrices.
Recall that $A_{i,j}$ is the locus of pairs of points which intersect a fixed general $i$-plane and are contained in a fixed general $j$-plane containing the $i$-plane.
In order for a pair of points to lie in both $A_{i,j}$ and $A_{k,l}$, we must have that $i+l\geq n$ and $j+k\geq n$.
Since $i+j+k+l = 2n$ (as they are complementary codimension), we see that $i+l=n$ and $j+k=n$.
So $A_{i,j}$ and $A_{k,l}$ are disjoint, which gives \[A_{i,j} \cdot A_{k,l} = 0,\] unless $k = n-j$ and $l=n-i$.

\medskip\noindent
For $A_{i,j}\cdot A_{n-j,n-i}$,
we see that the $i$-plane and the $(n-j)$-plane are disjoint.
Thus, the only pair of points in the intersection is the pair where one point is the intersection of the $i$-plane and the $(n-i)$-plane and the other point is the intersection of the $j$-plane and the $(n-j)$-plane.
In order to show that the intersection is $1$, it suffices to show that the intersection is transverse.
Since the pair of points is reduced, there are local charts for the Hilbert scheme of the form $\mathbb{A}^n \times \mathbb{A}^n$.
In these charts, the tangent spaces to each cycles are $\mathbb{A}^i \times \mathbb{A}^j$ and $\mathbb{A}^{n-j} \times \mathbb{A}^{n-i}$.
Since the two cycles are given by general fixed flags, we know that the elements of each flag are transverse.
It follows, that the intersection is transverse as needed.

\medskip\noindent 
Similarly, recall that $B_{i,j}'$ is the locus of pairs of points which intersect a fixed general $i$-plane and a fixed general $j$-plane intersecting the $i$-plane in a dimension $(i-1)$-plane and contained in their union.
In order for a pair of points to lie in both $B_{i,j}'$ and $B_{k,l}'$, we must have that $i+l\geq n$ and $j+k\geq n$.
The same reasoning as the previous gives \[B_{i,j}' \cdot B_{k,l}' = 0,\] unless $k = n-j$ and $l=n-i$.
Similarly, in the case of $B_{i,j}'\cdot B_{n-j,n-i}'$ with $i\neq j$, the same reasoning as the previous case shows they intersect transversely at a single pair.
In the case of $B_{i,i} \cdot B_{n-i,n-i}'$, there is a choice of which $i$-plane intersections which $(n-i)$-plane so the same reasoning as the previous case shows that they intersect transversely at two pairs of points.

\medskip\noindent 
Now consider $A_{i,j}\cdot B_{k,l}'$.
In order for a pair of points to lie in both $A_{i,j}$ and $B_{k,l}'$, we must have that $i+l\geq n$ and $j+k\geq n$.
The same reasoning as the previous case gives \[A_{i,j} \cdot B_{k,l}' = 0,\] unless $k = n-j$ and $l=n-i$.
Similarly, in the case of $A_{i,j}\cdot B_{n-j,n-i}'$, the same reasoning as the previous case shows they intersect transversely at a single pair since $i\neq j$.

\medskip\noindent
Finally, recall that $C_{k,l}$ is the locus of pairs of points which are linearly dependent with $k$ points and are contained in an $l$-plane which contains those points.
In order for a pair to lie in both $B_{i,j}'$ and $C_{k,l}'$, we must have that $i+l\geq n$ and $j+k\geq n$.
The same reasoning as before gives \[B_{i,j}' \cdot C_{k,l} = 0,\] unless $k = n-j$ and $l=n-i$.
In the case of $C_{i,j}\cdot B_{n-j,n-i}'$, we again see that there is a single pair in the intersection which consists of the point of intersection of the $j$-plane with the $(n-j)$-plane and the point of intersection of the $i$-plane with the $(n-i)$-plane.
%
Since the $j$-plane defining $C_{i,j}$ and the $(n-j)$-plane defining $B_{n-j,n-i}'$ are general, the $i$-plane determined by $C_{i,j}$ and the point of intersection is general.
Given that fact and that the $(n-i)$-plane defining $B_{n-j,n-i}'$ is general, it follows that this intersection is transverse as well.
\end{proof}

\medskip\noindent
This lemma gives us the tools we need to prove Theorem \ref{thm: chern}. 

\begin{reptheorem}{thm: chern}
In terms of the MS basis, we have 
\[c_1\left(\mathcal{O}_{\mathbb{P}^n}(d)^{[2]}\right) = (d-1)A_{n-1,n}+C_{n-1,n} \text{ and}\] 
\[c_2\left(\mathcal{O}_{\mathbb{P}^n}(d)^{[2]}\right) =\binom{d}{2}B_{n-1,n-1}'+dC_{n-1,n-1}\]
for $n>1$.
If $n=1$, the classes are the same without the second terms.
In particular, they are effective for all $n$.
\end{reptheorem}

\begin{proof}
For the $n>2$, case, it suffices to show last line of the tables in Table \ref{table: taut intersections} (the other lines follow from the previous lemmas).
\begin{table}
    \centering
\[\begin{array}{c|c|c|}
     & A_{0,1} & B'_{0,1}\\ \hline
    A_{n-1,n} & 1 & 1\\ \hline
    C_{n-1,n} & 0 & 1\\ \hline\hline
    c_1\left(\mathcal{O}_{\mathbb{P}^n}(d)^{[2]}\right) & d-1 & d \\ \hline
\end{array}\;\;\;\;
\begin{array}{c|c|c|c|c|}
     & A_{0,2} & B'_{0,2} & B_{1,1} & C_{1,1}\\ \hline
    A_{n-2,n} & 1  & 1 & 0 & 0\\ \hline
    B'_{n-1,n-1} & 0  & 0 & 2 & 1\\ \hline
    C_{n-1,n-1} & 0  & 0 & 1 & 0\\ \hline
    C_{n-2,n} & 0  & 1 & 0 & 0\\ \hline\hline
    c_2\left(\mathcal{O}_{\mathbb{P}^n}(d)^{[2]}\right) & 0  & 0 & d^2 & \binom{d}{2}\\ \hline
\end{array}\]
    \caption{The intersection matrices for the the tautological Chern classes and the MS basis in dimensions $1$ and $2$ with the MS basis in codimensions $1$ and $2$.}
    \label{table: taut intersections}
\end{table}

\medskip\noindent
In this case, the first Chern class is the class of the locus of pairs of points which lie on a member of a general pencil of hypersurfaces of degree $d$ and the the first second class is the class of the locus of pairs of points which lie on a general hypersurface of degree $d$. For more background on this, see \cite{MeS}.

\medskip\noindent
Recall that $A_{0,1}$ is the pairs which contain a point and are contained in a general line containing that point.
Given a general pencil of degree $d$ hypersurfaces, containing that fixed point determines a single member of that pencil.
Since the pencil, point, and line are general, that member of the pencil intersects the line transversely at $d$ distinct points.
Since the hypersurface and line intersect transversely at the fixed point, the second point of the pair must be one of the other $d-1$ points of intersection.
Thus, the intersection occurs at $d-1$ pairs of points. 
The transversality at each of those points follows from the transversality of the line and the determined member of the pencil. Thus. \[c_1\left(\mathcal{O}_{\mathbb{P}^n}(d)^{[2]}\right) \cdot A_{0,1} = d-1.\]

\medskip\noindent
Recall that $B'_{0,1}$ is the pairs which contain a point and intersect a general line (which by its generality does not contain that fixed point).
Again, given the general pencil of degree $d$ hypersurfaces, containing that fixed point determines a single member of that pencil.
Since the pencil, point, and line are general, that member of the pencil intersects the line transversely at $d$ distinct points.
The second point of the pair must be one of the $d$ points of intersection.
Thus, the intersection occurs at $d$ pairs of points. 
The transversality at each of those points follows from the transversality of the line and the determined member of the pencil. Thus. \[c_1\left(\mathcal{O}_{\mathbb{P}^n}(d)^{[2]}\right) \cdot B'_{0,1} = d,\]
and we have determined the first Chern class.

\medskip\noindent
Recall that $A_{0,2}$ is the pairs which contain a point and are contained in a general plane containing that point.
A general hypersurface of degree $d$ doesn't contain the fixed point so
\[c_2\left(\mathcal{O}_{\mathbb{P}^n}(d)^{[2]}\right) \cdot A_{0,2} = 0.\]

\medskip\noindent
Similarly, recall that $B'_{0,2}$ is the pairs which contain a point and are intersect a general plane.
A general hypersurface of degree $d$ doesn't contain the fixed point so
\[c_2\left(\mathcal{O}_{\mathbb{P}^n}(d)^{[2]}\right) \cdot B'_{0,2} = 0.\]

\medskip\noindent
Again similarly, recall that $B'_{1,1}$ is the pairs which intersect each of a pair of intersecting lines and are contained in their union.
A general hypersurface of degree $d$ intersects each of those lines transversally at $d$ distinct points.
Any pair in the intersection must contain one of each set of intersection points.
Thus, the intersection occurs at $d^2$ pairs of points. 
The transversality at each of those points follows from the transversality of the lines and the hypersurface. Thus. \[c_2\left(\mathcal{O}_{\mathbb{P}^n}(d)^{[2]}\right) \cdot B'_{1,1} = d^2.\]

\medskip\noindent
Finally, recall that $C_{1,1}$ is the pairs which are contained in a general line.
A general hypersurface of degree $d$ intersects that line transversally at $d$ distinct points.
Any pair in the intersection must contain two of these intersection points.
Thus, the intersection occurs at $\binom{d}{2}$ pairs of points. 
The transversality at each of those points follows from the transversality of the line and the hypersurface. Thus. \[c_2\left(\mathcal{O}_{\mathbb{P}^n}(d)^{[2]}\right) \cdot C_{1,1} = \binom{d}{2},\]
and we have computed the second Chern class.

\medskip\noindent
Note, in the $n=1$ and $n=2$, only some of these intersections make sense as only some of the basis classes exist, but those that exist have the same intersections by the same arguments. For $n=2$, this does not change the class, but for $n=1$ it changes the class by dropping the second terms, as those basis elements do not exist.
\end{proof}

\section{Application to secant varieties}
\noindent 
In this section, we apply the results of the previous section to compute the degree of the (1st) secant variety of a complete intersection subvariety of projective space which is not $1$-defective. 
These methods are likely applicable to other related questions, such as the number of secants to a complete intersection of dimension $n$ through a general point of $\mathbb{P}^{2n+1}$.

\medskip\noindent
We first recall some terminology around secant varieties.
We will define the notions only for the secant variety of lines, but they are defined in full generality for the secant variety of $k$-planes.
The $1$\textit{-th secant variety}, $\mathrm{Sec}_1(X) =\mathrm{Sec}(X)$, is the (closure of the) union of lines that intersect $X$ at $2$ or more points.
The expected dimension of $\mathrm{Sec}(X)$ is $2\dim(X)+1$, and a variety is called $1$-defective if $\dim\left(\mathrm{Sec}(X)\right)<2\dim(X)+1$.
The $1$\textit{-th secant order} of a variety, which is not $1$-defective, is the number of secant lines through a general point of the $\mathrm{Sec}(X)$; denote this by $\mu_1(X)$.

\begin{reptheorem}{thm: sec}
Let $X$ be a subvariety of $\mathbb{P}^n$ of dimension $m$ with $2m+1<n$ which is not $1$-defective and which is the complete intersection of hypersurfaces of degree $d_1$, $\dots$, and $d_{n-m}$.
The degree of $\mathrm{Sec}(X)$ is \[\deg\left(\mathrm{Sec}(X)\right) = \frac{1}{\mu_1(X)} \sum_{k=m+1}^{n-m} \left(\prod_{j \in \{i_1,\dots,i_k\}} \left(\prod_{l\not \in \{i_1,\dots,i_k\}} \binom{d_{j}}{2} d_l 2^{k-1-m}\right)\right)\] 
where the products are taken over all choices of $k$ integers from the set $\{1,\dots,n-m\}$.
\end{reptheorem}
\noindent
In order to prove this result, we first need the following lemma.

\begin{lemma}
\label{lem: products}
We have the following intersections
\begin{align*}
    B_{n-1,n-1}^{'k} &= 2^{k-1}\left(B_{n-k,n-k}'+ \sum_{i=1}^{k-1}    \left(B_{n-k-i,n-k+i}'-A_{n-k-i,n-k+i}\right)\right) \text{ if } 0< 2k-1\leq n\\
    B_{n-1,n-1}^{'k}&= 2^{k-1}\left(B_{n-k,n-k}'+ \sum_{i=1}^{n-k}    \left(B_{n-k-i,n-k+i}'-A_{n-k-i,n-k+i}\right)\right) \text{ if }  n\leq  2k-1<2n
\end{align*}
\end{lemma}

\medskip\noindent
\begin{proof}[Proof of Lemma \ref{lem: products}]
In this proof, fix $F_*$, $F_*'$, and $F_*''$ to be three distinct flags of linear subspaces of $\mathbb{P}^n$.

\medskip\noindent
We proceed by induction on $k$ to prove the first formula, which we will then use to prove the second formula.
In the case that $k=1$, the first formula reduces to
\[B_{n-1,n-1}' = 2^{1-1}\left(B_{n-1,n-1}'+ \sum_{i=1}^{0}    \left(B_{n-1-i,n-1+i}'-A_{n-1-i,n-1+i}\right)\right) = B_{n-1,n-1}'\]
so the base case trivially holds.

\medskip\noindent
For the induction step, we need the intersections 
$B_{n-1,n-1}'\cdot A_{n-k-i,n-k+i}$ for $1 \leq i \leq k$ and $B_{n-1,n-1}'\cdot B_{n-k-i,n-k+i}'$ for $0 \leq i \leq k$.

\medskip\noindent
Any pair $Z$ in the intersection $B_{n-1,n-1}'\cdot A_{n-k-i,n-k+i}$ needs to satisfy the five conditions: $\mathrm{supp}(Z)\cap F_{n-1} \neq \varnothing$, $\mathrm{supp}(Z)\cap F_{n-1}' \neq \varnothing$, $Z \subset F_{n-1}\cup F_{n-1}'$, $\mathrm{supp}(Z)\cap F_{n-k-i}''\neq \varnothing$,  and $Z \subset F_{n-k+i}''$.
From the first four conditions, we can see that either $\mathrm{supp}(Z) \cap (F_{n-1}\cap F_{n-1-i}'') \neq \varnothing$ or $\mathrm{supp}(Z) \cap (F_{n-1}'\cap F_{n-k-i}'') \neq \varnothing$.
In the first case, for $Z$ to be in the intersection it suffices for $Z \cap (F_{n-1}'\cap F_{n-k+i}'') \neq \varnothing$ and $Z \subset (F_{n-1}\cap F_{n-k-i}'')\cup (F_{n-1}'\cap F_{n-k+i}'')$.
Since these are linear subspaces of $\mathbb{P}^n$ of dimension $n-1-k-i$ and $n-1-k+i$ which overlap in the subspace $F_{n-1}\cap F_{n-1}' \cap F_{n-k-i}''$ of dimension $n-2-k-i$, the intersection is some multiple of the class $B'_{n-1-k-i,n-1-k+i}$.
However, the transversality of the linear subspaces involved, since they are general, implies that this intersection is transverse so the multiple is just 1. 
The other case is similar so
\[B_{n-1,n-1}'\cdot A_{n-k-i,n-k+i} = 2B'_{n-1-k-i,n-1-k+i}.\]

\medskip\noindent
The second type of intersection is much more difficult to see geometrically so we wish to write $B_{n-1-i,n-1+i}'$ in the BB basis.
By Tables \ref{fig: MS/ES intersections} and \ref{table: MS intersections}, it follows that if $1 \leq i \leq k$ \[B_{n-k-i,n-k+i} = l\left(B_{n-k-i,n-k+i}'-A_{n-k-i,n-k+i}\right)\] with $l = C_{k-i,k+i}\cdot B_{n-k-i,n-k+i}$.
Any pair $Z$ in that intersection needs to satisfy the five conditions: $\dim\left(\mathrm{Span}(Z,F_{k-i-1})\right) \leq k-i$, $Z\subset F_{k+i}$, $\mathrm{supp}(Z) \subset F'_{n-k-i}$, $Z \subset F'_{n-k+i+1}$, and $Z$ is non-reduced.
From the second and third conditions, we can see that $Z$ must be supported on $F_{k+i} \cap F'_{n-k-i}$, which is a single point so the support of any such $Z$ is fixed.
In particular, this fixes the $(k-i)$-dimensional space $\mathrm{Span}(Z,F_{k-i-1})$ which must contain $Z$.
Then $Z$ must be contained in the $\mathrm{Span}(Z,F_{k-i-1}) \cap F_{n-k+i-1}$, which is a single line containing the point where $Z$ must be supported.
There is a single non-reduced scheme at a fixed point which is contained in a line so $Z$ is unique.
A computation is local coordinates shows that this multiplicity is 2.
Thus, $l=2$ and \[B_{n-k-i,n-k+i} = 2\left(B_{n-k-i,n-k+i}'-A_{n-k-i,n-k+i}\right).\]

\medskip\noindent
Similarly, $B_{n-k,n-k} = l\left(B_{n-k,n-k}'-2C_{n-k,n-k}\right)$ for some positive integer $l = C_{k,k}\cdot B_{n-k,n-k}$.
An almost identical analysis again yields $l=2$ and
\[B_{n-k,n-k} = 2\left(B_{n-k,n-k}'-2C_{n-k,n-k}\right)\]

\medskip\noindent
Solving for $B_{n-k-i,n-k+i}'$, gives $B_{n-k-i,n-k+i}' = \frac{1}{2}B_{n-k-i,n-k+i}+A_{n-k-i,n-k+i} $ for $1\leq i\leq n$ and $B_{n-k,n-k}'= \frac{1}{2}B_{n-k,n-k}+2C_{n-k,n-k}$.
Since we already know $B_{n-1,n-1}' \cdot A_{n-k-i,n-k+i}$ it suffices to know $B_{n-1,n-1}'\cdot B_{n-k-i,n-k+i}$ for $0\leq i\leq k$ and $B_{n-1,n-1}'\cdot C_{n-k,n-k}$.

\medskip\noindent
Any pair $Z$ in the intersection $B_{n-1,n-1}'\cdot B_{n-k-i,n-k+i}$ needs to satisfy the six conditions: $\mathrm{supp}(Z)\cap F_{n-1} \neq \varnothing$, $\mathrm{supp}(Z)\cap F_{n-1}' \neq \varnothing$, $Z \subset F_{n-1}\cup F_{n-1}'$, $\mathrm{supp}(Z) \subset F''_{n-k-i}$, $Z \subset F''_{n-k+i+1}$, and $Z$ is non-reduced.
Since $Z$ is non-reduced, it must be supported on the intersection of $F_{n-1}\cap F_{n-1}' \cap F_{n-k-i}''$. 
Since the span of $F_{n-1}$ and $F_{n-1}'$ is the whole space, any double point supported on their intersection in contained in them (and intersects them both).
So for $Z$ to be in the intersection it suffices for it to be non-reduced, supported on $F_{n-1}\cap F_{n-1}' \cap F_{n-k-i}''$, and contained in $F''_{n-k+i+1}$.
Since $F_{n-1}\cap F_{n-1}' \cap F_{n-k-i}''$ is a $(n-k-i-2)$-dimensional space contained inside $F''_{n-k+i+1}$, this means the intersection is some multiple of $B_{n-k-i-2,n-k+i}$.
This multiplicity is given by $C_{k-i,k+i+2}\cdot B_{n-1,n-1}'\cdot B_{n-k-i,n-k+i}$.
A computation in local coordinates shows that this is 2.
Thus, \[B_{n-1,n-1}'\cdot B_{n-k-i,n-k+i} = 2B_{n-k-i-2,n-k+i}.\]

\medskip\noindent
Any pair $Z$ in the intersection $B_{n-1,n-1}'\cdot C_{n-k,n-k}$ needs to satisfy the four conditions: $\mathrm{supp}(Z) \cap F_{n-1} \neq \varnothing$, $\mathrm{supp}(Z)\cap F_{n-1}' \neq \varnothing$, $Z \subset F_{n-1}\cup F_{n-1}'$, and $Z \subset F''_{n-k}$.
Combining the fourth condition into the other lets us restate these as $\mathrm{supp}(Z)\cap F_{n-1}\cap F''_{n-k} \neq \varnothing$, $\mathrm{supp}(Z)\cap F_{n-1}'\cap F''_{n-k} \neq \varnothing$, $Z \subset (F_{n-1}\cap F''_{n-k})\cup (F_{n-1}'\cap F''_{n-k})$.
Since $F_{n-1}\cap F''_{n-k}$ and $F_{n-1}'\cap F''_{n-k} $ are dimension  $(n-k-1)$ and intersect in dimension $(n-k-2)$, this is up to multiple the class $B_{n-k-1,n-k-1}'$.
Again, the transversality of the linear spaces shows the transversality of the intersection. Thus,
\[B_{n-1,n-1}'\cdot C_{n-k,n-k} = B_{n-k-1,n-k-1}'.\]

\medskip\noindent
Putting those intersections together gives 
\begin{align*}
    B_{n-1,n-1}'\cdot B_{n-k,n-k}' 
    &=2\left(B_{n-k-2,n-k}'-A_{n-k-2,n-k}+B_{n-k-1,n-k-1}'\right)\;\; \text{ and }\\
    B_{n-1,n-1}'\cdot B_{n-k-i,n-k+i}' 
    &=2\left(B_{n-k-i-1,n-k-i-1}'+B_{n-k-i-2,n-k}'-A_{n-k-i-2,n-k}\right).
\end{align*}
With these pairwise intersections calculated, the first formula follows by several substitutions and one re-indexing of a summation. 

\medskip\noindent
The second formula is proved by induction on $k-\lfloor\frac{n+1}{2}\rfloor$.
The base case follows is an example of the first formula.
The induction step uses the same pairwise intersections as the previous step but noting that any cycle whose lower index drops below $0$ is a zero cycle.
\end{proof}

\medskip\noindent
With the lemma in hand, we can proceed with the proof of the theorem.

\begin{proof}[Proof of Theorem \ref{thm: sec}]
We first notice that the degree of the secant variety times the $1$-st secant order of the variety is the number of points on a general dimension $n-2m-1$ linear subspace in $\mathbb{P}^n$ which lie on a secant line to the variety.
This is equivalent to the number of pairs of points on $X$ which are linearly dependent with the linear space.
This can be computed as the intersection of the locus of pairs of points on $X$, which we denote $Y$, with the class $C_{n-2m,n}$.
Since $X$ is a complete intersection, a pair of points lying on $X$ is equivalent to them lying on each hypersurface cutting it out and each one of those conditions cuts the dimension of $Y$ by $2$.
In other words, $Y = \prod_{i=1}^{n-m} c_2\left(\mathcal{O}_{\mathbb{P}^n}(d_i)\right)$.
Thus, \[\deg\left(\mathrm{Sec}(X)\right)\mu_1(X) = C_{n-2m,n}\cdot \prod_{i=1}^{n-m} c_2\left(\mathcal{O}_{\mathbb{P}^n}(d_i)\right).\] 

\medskip\noindent
Each term of that product is of the form $\binom{d_i}{2}B_{n-1,n-1}'+d_i C_{n-1,n-1}$. 
If $2k-1 \leq n$, then this product has the form 
\[B_{n-1,n-1}^{'k}\cdot C_{n-1,n-1}^{n-m-k} = 2^{k-1}\left(B_{n-k,n-k}'+ \sum_{i=1}^{k-1}    \left(B_{n-k-i,n-k+i}'-A_{n-k-i,n-k+i}\right)\right)\cdot C_{n-1,n-1}^{n-m-k} \]
By a similar argument to the proof of previous lemma, $A_{n-k-i,n-k+i}\cdot C_{n-1,n-1} = A_{n-k-i-1,n-k+i-1}$.
Using this product and the product $B_{n-k-i,n-k+i}'\cdot C_{n-1,n-1}$ from the proof of the previous lemma, it follows by a simple induction that this reduces to
\[B_{n-1,n-1}^{'k}\cdot C_{n-1,n-1}^{n-m-k} = 2^{k-1}\left(B_{m,m}'+ \sum_{i=1}^{k-1} \left(B_{m-i,m+i}'-A_{m-i,m+i}\right)\right).\]
If $n< 2k-1 <2n$, we similarly get that it reduces to
\[B_{n-1,n-1}^{'k}\cdot C_{n-1,n-1}^{n-m-k} = 2^{k-1}\left(B_{m,m}'+ \sum_{i=1}^{n-k}  \left(B_{m-i,m+i}'-A_{m-i,m+i}\right)\right).\]

\medskip\noindent
We are interested in the triple product of $C_{n-2m,n}$ with each of those products.
Since $C_{n-2m,n}$ intersects every complementary codimension class in the MS basis except $B_{0,2m}'$ in zero and intersects that class in 1, 
\[C_{n-2m,n}\cdot B_{n-1,n-1}^{'k}\cdot C_{n-1,n-1}^{n-m-k} =
\begin{cases}
2^{k-1}B_{0,0}' &\text{if }0<2k-1\leq n \text{ and } m\leq k-1\\
2^{k-1}B_{0,0}' &\text{if }n<2k-1< 2n \text{ and } m\leq n-k\\
0 &\text{otherwise}
\end{cases}\]
This can be further simplified to
\[C_{n-2m,n}\cdot B_{n-1,n-1}^{'k}\cdot C_{n-1,n-1}^{n-m-k} =
\begin{cases}
2^{k-1}B_{0,0}' &\text{if }m+1 \leq k \leq n-m\\
0 &\text{otherwise}
\end{cases}\]


\medskip\noindent
We want to use these triple products in our original product \[C_{n-2m,n}\cdot \prod_{i=1}^{n-m} c_2\left(\mathcal{O}_{\mathbb{P}^n}(d_i)\right).\]
which expands to
\[\sum_{k=0}^{n-m} \prod_{j \in \{i_1,\dots,i_k\}} \prod_{l\not \in \{i_1,\dots,i_k\}}\binom{d_{j}}{2} d_l B_{n-1,n-1}^{'k}\cdot C_{n-1,n-1}^{n-m-k}\cdot C_{n-2m,n}.\]
Using the formula for the triple intersections, we finally get
\[\deg\left(\mathrm{Sec}(X)\right)\mu_1(X) =\sum_{k=m+1}^{n-m} \left(\prod_{j \in \{i_1,\dots,i_k\}} \left(\prod_{l\not \in \{i_1,\dots,i_k\}} \binom{d_{j}}{2} d_l 2^{k-1}\right)\right).\]

\end{proof}

\bibliographystyle{alphanum}
\bibliography{main}
\end{document}